\documentclass[preprint, 11pt]{article}

\pdfoutput=1

\usepackage{amsmath,amsthm,amsfonts,amssymb,amscd}
\usepackage{graphicx,bm,color}
\usepackage{algorithm}
\usepackage{subcaption}
\usepackage[T1]{fontenc}
\usepackage{algpseudocode}
\usepackage{stmaryrd}
\usepackage{listings}
\usepackage{lineno}
\usepackage{mathptmx}
\usepackage{cleveref}
\usepackage[skins,theorems]{tcolorbox}
\usepackage{authblk}
\usepackage[latin1]{inputenc}
\usetikzlibrary{arrows,calc, decorations.markings, intersections}
\usepackage{caption}
\usepackage{cancel}
\usepackage{mathtools}
\usepackage{url}
 
\definecolor{gray}{gray}{0.6}

\theoremstyle{plain}

\newtheorem{theorem}{Theorem}

\newtheorem{lemma}[theorem]{Lemma}

\theoremstyle{remark}

\newtheorem*{remark*}{Remark}

\newcommand{\vertiii}[1]{{\left\vert\kern-0.25ex\left\vert\kern-0.25ex\left\vert #1 
		\right\vert\kern-0.25ex\right\vert\kern-0.25ex\right\vert}}

\newcommand{\Mod}[1]{\ (\mathrm{mod}\ #1)}

\begin{document}

\title{A proof of An{\dj}eli\'c-Fonseca conjectures on the determinant of some Toeplitz matrices and their generalization }

\author[1]{Bakytzhan Kurmanbek}
\author[1, *]{Yerlan Amanbek}
\author[2]{Yogi Erlangga}
\affil[1]{Nazarbayev University, Department of Mathematics, 53 Kabanbay Batyr Ave, Nur-Sultan 010000, Kazakhstan}
\affil[2]{Zayed University, Department of Mathematics, Abu Dhabi Campus, P.O. Box 144534, United Arab Emirates}
\affil[ ]{\textit{ bakytzhan.kurmanbek@nu.edu.kz, yerlan.amanbek@nu.edu.kz, yogi.erlangga@zu.ac.ae}}
\affil[*]{Corresponding author}

\date{\today}
\maketitle

\begin{abstract}

We present a proof of determinant of special nonsymmetric Toeplitz matrices conjectured by An{\dj}eli\'c and Fonseca in \cite{andjelic2020some}. A proof is also demonstrated for a more general theorem. The two conjectures are therefore just two possible results, under two specific settings. Numerical examples validating the theorem are provided. 

\end{abstract}




\newcommand{\bs}[1]{\boldsymbol{#1}}

\section{Introduction}

In this paper, we focus on the determinant of $n \times n$ Toeplitz matrices of the form
\begin{equation}
\begin{pmatrix}
1 & 1 & 0 & 0 &  \cdots & a & b\\
1 & 1 & 1 & 0 & \cdots & 0 & a \\
1 & 1 & 1 & 1 &  \cdots & 0 & 0 \\ 
0 & 1 & 1 & 1 & \cdots & 0 & 0 \\
\vdots &   & \ddots & \ddots &  \ddots & \ddots &  \vdots \\
0 & \cdots & \cdots & 1 & 1 & 1 & 1 \\
0 & \cdots & \cdots & 0 & 1 & 1 & 1 \\
\end{pmatrix}, \label{genmatrix}
\end{equation}
where $a, b \in \mathbb{R}$. An{\dj}eli\'c and Fonseca consider two cases: (i)  $a=0$ and $b=1$; and (ii) $a=1$ and $b=0$, and give an explicit formula of the determinant stated as Conjecture~6.1 and~6.2 in \cite{andjelic2020some}. We shall present a proof of the two conjectures (in Section 2) and generalize the statement for any $a,b$ (Section 3). Numerical examples of determinant of $n \times n$ matrices would be shown (Section 4). The proof is mostly based on a tedious use of elementary techniques.

To make this paper self-contained, we state the two conjectures of  \cite{andjelic2020some} below, now as theorem. Throughout the paper, the notation $| \cdot |_n$ denotes the determinant of an $n \times n$ matrix in the argument; i.e., $| \cdot |_n = \det(\cdot)$, where ``$\cdot$'' is an $n \times n$ matrix.

\begin{theorem}[Conjecture 6.1 of \cite{andjelic2020some}] \label{thm:conj1}
	Let $A_n$ be the determinant of $n \times n$ matrix given by
	\begin{equation*}
	A_{n} = 
	\begin{vmatrix}
	1 & 1 & 0 & 0 & 0 & \cdots & 1 \\
	1 & 1 & 1 & 0 & 0 & \cdots & 0 \\
	1 & 1 & 1 & 1 & 0 & \cdots & 0 \\ 
	0 & 1 & 1 & 1 & 1 & \cdots & 0 \\
	\vdots &   & \ddots & \ddots &  \ddots & \ddots &  \vdots \\
	0 & \cdots & \cdots & 1 & 1 & 1 & 1 \\
	0 & \cdots & \cdots & 0 & 1 & 1 & 1 \\
	\end{vmatrix}_n.
	\end{equation*}
	\newline
	Then 
	
	\begin{equation*}
	A_n=
	\begin{cases}
	1, & \text{if } n \equiv 0 \Mod{4}, \\  
	2, & \text{if } n \equiv 1 \Mod{4}, \\
	-1,& \text{if } n \equiv 2 \Mod{4}, \\
	0,  &\text{if } n \equiv 3 \Mod{4} .
	\end{cases}
	\end{equation*}	
\end{theorem}

\begin{theorem}[Conjecture 6.2 of \cite{andjelic2020some}]  \label{thm:conj2}
	Let $M_n$ be the determinant of $n \times n$ matrix given by 
	\begin{equation*} \label{matM}
	M_{n} = 
	\begin{vmatrix}
	1 & 1 & 0 &   0 & 0  & 1  & 0\\
	1 & 1 & 1 &   0 & \ddots &  \ddots & 1  \\
	1 & 1 & 1 &    1 & \ddots  & 0 & 0 \\ 
	0 & 1 & 1 &    \ddots & \ddots & 0 & \vdots \\
	\vdots & \ddots  & \ddots & \ddots &    \ddots &  \ddots &  0\\
	\vdots &   & \ddots &  1 & 1 & 1 & 1\\
	0 & \cdots & \cdots &   0  & 1 & 1 & 1\\
	\end{vmatrix}_n.
	\end{equation*}
	\newline
	Then 
	\begin{equation*}
	M_n=
	\begin{cases}
	0, &\text{  if  } n \equiv 0 \Mod{4} \\  
	2, &\text{  if  } n \equiv 1 \Mod{4} \\
	3, &\text{  if  } n \equiv 2 \Mod{4} \\
	1, &\text{  if  } n \equiv 3 \Mod{4} 
	\end{cases}
	\end{equation*}
\end{theorem}

\section{Proof of conjectures} \label{sec:proof}

To prove the conjectures, we require few lemmas, whose proof are derived in the subsequent sections. We shall first prove Theorem~\ref{thm:conj1}, followed by Theorem~\ref{thm:conj2}.

\subsection{Proof of Theorem \ref{thm:conj1}}

We shall first prove the following lemma.

\begin{lemma}\label{lem:main}
	Let
	\begin{equation*}
	C_{n} = 
	\begin{vmatrix}
	1 & 1 & 0 & \cdots & \cdots & \cdots & 0 \\
	1 & 1 & 1 & 0 & \cdots & \cdots & \vdots \\
	1 & 1 & 1 & 1 & 0 & \cdots & \vdots \\ 
	0 & 1 & 1 & 1 & 1 & \ddots & \vdots \\
	\vdots & \ddots  & \ddots & \ddots &  \ddots & \ddots &  0 \\
	\vdots & \cdots & \ddots & 1 & 1 & 1 & 1 \\
	0 & \cdots & \cdots & 0 & 1 & 1 & 1 \\
	\end{vmatrix}_n .
	\end{equation*}
	\newline
	Then  
	\begin{equation*}
	C_{n} = C_{n-4} \quad  \text{for $n \ge 5$}.
	\end{equation*}
\end{lemma}

\begin{proof}
	Cofactor expansion across the first row yields
	\begin{equation}
	\begin{aligned}
	C_{n} &= (-1)^{1+1} \begin{vmatrix}
	1 & 1 & 0 & \cdots & \cdots & \cdots & 0 \\
	1 & 1 & 1 & 0 & \cdots & \cdots & \vdots \\
	1 & 1 & 1 & 1 & 0 & \cdots & \vdots \\ 
	0 & 1 & 1 & 1 & 1 & \ddots & \vdots \\
	\vdots & \ddots  & \ddots & \ddots &  \ddots & \ddots &  0 \\
	\vdots & \cdots & \ddots & 1 & 1 & 1 & 1 \\
	0 & \cdots & \cdots & 0 & 1 & 1 & 1 \\
	\end{vmatrix}_{n-1}
	+(-1)^{1+2}\begin{vmatrix}
	1 & 1 & 0 & \cdots & \cdots & \cdots & 0 \\
	1 & 1 & 1 & 0 & \cdots & \cdots & \vdots \\
	0 & 1 & 1 & 1 & 0 & \cdots & \vdots \\ 
	0 & 1 & 1 & 1 & 1 & \ddots & \vdots \\
	\vdots & \ddots  & \ddots & \ddots &  \ddots & \ddots &  0 \\
	\vdots & \cdots & \ddots & 1 & 1 & 1 & 1 \\
	0 & \cdots & \cdots & 0 & 1 & 1 & 1 \\
	\end{vmatrix}_{n-1} \\
	&=:  C_{n-1} - T_{n-1}.
	\end{aligned} \notag
	\end{equation}
     From now on, we will write "expansion" instead of "cofactor expansion". We apply the expansion across the first row for matrix $T_{n-1}$, we have
	\begin{equation}
	\begin{aligned}
	T_{n-1} &= (-1)^{1+1} \begin{vmatrix}
	1 & 1 & 0 & \cdots & \cdots & \cdots & 0 \\
	1 & 1 & 1 & 0 & \cdots & \cdots & \vdots \\
	1 & 1 & 1 & 1 & 0 & \cdots & \vdots \\ 
	0 & 1 & 1 & 1 & 1 & \ddots & \vdots \\
	\vdots & \ddots  & \ddots & \ddots &  \ddots & \ddots &  0 \\
	\vdots & \cdots & \ddots & 1 & 1 & 1 & 1 \\
	0 & \cdots & \cdots & 0 & 1 & 1 & 1 \\
	\end{vmatrix}_{n-2} + (-1)^{1+2}\begin{vmatrix}
	1 & 1 & 0 & \cdots & \cdots & \cdots & 0 \\
	0 & 1 & 1 & 0 & \cdots & \cdots & \vdots \\
	0 & 1 & 1 & 1 & 0 & \cdots & \vdots \\ 
	0 & 1 & 1 & 1 & 1 & \ddots & \vdots \\
	\vdots & \ddots  & \ddots & \ddots &  \ddots & \ddots &  0 \\
	\vdots & \cdots & \ddots & 1 & 1 & 1 & 1 \\
	0 & \cdots & \cdots & 0 & 1 & 1 & 1 \\
	\end{vmatrix}_{n-2} \\
	&=:  C_{n-2}-C_{n-3}.
	\end{aligned}  \notag
	\end{equation}
	Thus, $C_{n} = C_{n-1} - C_{n-2}+C_{n-3}$, and hence $ C_{n-1}= C_{n-2} - C_{n-3}+C_{n-4}$. By substituting $C_{n-1}$ we have 
	
	$$C_{n} = C_{n-2} - C_{n-3}+C_{n-4} - C_{n-2}+C_{n-3} = C_{n-4}$$
\end{proof}

The proof of Theorem \ref{thm:conj1} then proceeds as follows.

Expansion $A_{n}$ across the first row and applying some manipulations on the involved determinants as we did in the proof of Lemma~\ref{lem:main}, we get: 

\begin{equation}
\begin{aligned}
A_{n} &=  \begin{vmatrix}
1 & 1 & 0 & 0 & 0 & \cdots & 0 \\
1 & 1 & 1 & 0 & 0 & \cdots & 0 \\
1 & 1 & 1 & 1 & 0 & \cdots & 0 \\ 
0 & 1 & 1 & 1 & 1 & \cdots & 0 \\
\vdots &   & \ddots & \ddots &  \ddots & \ddots &  \vdots \\
0 & \cdots & \cdots & 1 & 1 & 1 & 1 \\
0 & \cdots & \cdots & 0 & 1 & 1 & 1 \\
\end{vmatrix}_{n-1} - 
\begin{vmatrix}
1 & 1 & 0 & 0 & 0 & \cdots & 0 \\
1 & 1 & 1 & 0 & 0 & \cdots & 0 \\
0 & 1 & 1 & 1 & 0 & \cdots & 0 \\ 
0 & 1 & 1 & 1 & 1 & \cdots & 0 \\
\vdots &   & \ddots & \ddots &  \ddots & \ddots &  \vdots \\
0 & \cdots & \cdots & 1 & 1 & 1 & 1 \\
0 & \cdots & \cdots & 0 & 1 & 1 & 1 \\
\end{vmatrix}_{n-1} \\
&+(-1)^{n+1}\begin{vmatrix}
1 & 1 & 1 & 0 & 0 & \cdots & 0 \\
1 & 1 & 1 & 1 & 0 & \cdots & 0 \\
0 & 1 & 1 & 1 & 1 & \cdots & 0 \\ 
\vdots &   & \ddots & \ddots &  \ddots & \ddots &  \vdots \\
0 & \cdots & \cdots & 1 & 1 & 1 & 1 \\
0 & \cdots & \cdots & 0 & 1 & 1 & 1 \\
0 & \cdots & \cdots & 0 & 0 & 1 & 1 \\
\end{vmatrix}_{n-1}
\\
&= C_{n-1} - (C_{n-2} - C_{n-3}) + (-1)^{n+1} C_{n-1} \\
&= C_{n} + (-1)^{n+1} C_{n-1}.
\end{aligned} \notag
\end{equation}
\newline
The first four $C_n$ are as follows:
$$ C_{1} = 1, \; C_{2} = \begin{vmatrix}
1 & 1 \\
1 & 1 
\end{vmatrix} = 0, \; C_{3} = \begin{vmatrix}
1 & 1 & 0 \\
1 & 1  & 1 \\
1 & 1 & 1
\end{vmatrix} = 0, \; C_4=\begin{vmatrix}
1 & 1 & 0 & 0 \\
1 & 1  & 1 & 0\\
1 & 1 & 1 & 1 \\
0 & 1 & 1 & 1 
\end{vmatrix} = 1. $$
\newline
Using Lemma \ref{lem:main} we can conclude that  $C_{4k+1}=1$, $C_{4k+2}=0$, $C_{4k+3}=0$ and $C_{4k+4}=1$ for $k=0,1,2,..,n$.
Therefore, the statement of the theorem  holds since $A_n=C_n+(-1)^{n+1}C_{n-1}$. 

\subsection{Proof of Theorem \ref{thm:conj2}}

To prove Theorem  \ref{thm:conj2}, we first prove the following two lemmas.

\begin{lemma}\label{lem:kn}
	Let
	\begin{equation*}
	K_{n} = 
	\begin{vmatrix}
	1 & 0 & 0 & \cdots & 0 & 1  & 0\\
	1 & 1 & 1 & \ddots & \ddots & 0 & 0  \\
	1 & 1 & \ddots &  \ddots & \ddots & \ddots & 0 \\ 
	0 & 1 & \ddots &  \ddots & \ddots & \ddots & \vdots \\
	\vdots &  \ddots  & \ddots &\ddots &  \ddots &  \ddots &  0\\
	0 & \cdots & 0 & 1 & 1 & 1 & 1\\
	0 & \cdots & \cdots & 0  & 1 & 1 & 1\\
	\end{vmatrix}_n
	\end{equation*}
	Then
	\newline 
	$$K_{n} = (C_{n-2} - C_{n-3})(1 + (-1)^{n}) + C_{n-4}.$$
\end{lemma}
\begin{proof}
	Expansion of $K_n$ along the last column yields
	\begin{equation}
	\begin{aligned}
	K_{n} &= \begin{vmatrix}
	1 & 0 & 0 & \cdots & 0 & 0  & 1\\
	1 & 1 & 1 & \ddots & \ddots & 0 & 0  \\
	1 & 1 & \ddots &  \ddots & \ddots & \ddots & 0 \\ 
	0 & 1 & \ddots &  \ddots & \ddots & \ddots & \vdots \\
	\vdots &  \ddots  & \ddots &\ddots &  \ddots &  \ddots &  0\\
	0 & \cdots & 0 & 1 & 1 & 1 & 1\\
	0 & \cdots & 0 & 0  & 1 & 1 & 1\\
	\end{vmatrix}_{n-1}
	- 
	\begin{vmatrix}
	1 & 0 & 0 &  \cdots & 0 & 0  & 1\\
	1 & 1 & 1 &  \ddots & \ddots & 0 & 0  \\
	1 & 1 & \ddots &  \ddots & \ddots & \ddots & 0 \\ 
	0 & 1 & \ddots & \ddots & \ddots & \ddots & \vdots \\
	\vdots & \ddots  & \ddots & \ddots &    \ddots &  \ddots &  0 \\
	0 & \cdots & 0  & 1 & 1 & 1 & 1\\
	0 & \cdots & \cdots   & 0  & 0 & 1 & 1\\
	\end{vmatrix}_{n-1} \\
	&=: R_{n-1} - P_{n-1}. 
	\end{aligned} \notag
	\end{equation}
	
	$R_{n-1}$ and $P_{n-1}$ are now expressible in terms of $C_n$ after the expansion across the first row: 
	\begin{equation}
	\begin{aligned}
	R_{n-1} &= \begin{vmatrix}
	1 & 1 & 0 & \cdots & \cdots & \cdots & 0 \\
	1 & 1 & 1 & 0 & \cdots & \cdots & \vdots \\
	1 & 1 & 1 & 1 & 0 & \cdots & \vdots \\ 
	0 & 1 & 1 & 1 & 1 & \ddots & \vdots \\
	\vdots & \ddots  & \ddots & \ddots &  \ddots & \ddots &  0 \\
	\vdots & \cdots & \ddots & 1 & 1 & 1 & 1 \\
	0 & \cdots & \cdots & 0 & 1 & 1 & 1 \\
	\end{vmatrix}_{n-2}  + (-1)^{n}
	\begin{vmatrix}
	1 & 1 & 1 & 0 & \cdots & \cdots & 0  \\
	1 & 1 & 1 & 1 & \ddots  & \cdots & \vdots \\
	0 & 1 & 1 & 1 & 1 & \ddots & \vdots  \\ 
	\vdots & 0 & 1 & 1 & 1 & \ddots & 0 \\
	\vdots & \ddots  & \ddots & \ddots &  \ddots & \ddots &  1\\
	\vdots & \cdots & \ddots & 0 & 1 & 1 & 1  \\
	0 & \cdots & \cdots & \cdots & 0 & 1 & 1  \\
	\end{vmatrix}_{n-2} \\
	&= C_{n-2} + (-1)^{n}C_{n-2}.
	\end{aligned} \notag
	\end{equation}
	
	\noindent and 
	\begin{equation}
	\begin{aligned}
	P_{n-1} &=\begin{vmatrix}
	1 & 1 & 0 & \cdots & \cdots & \cdots & 0 \\
	1 & 1 & 1 & 0 & \cdots & \cdots & \vdots \\
	1 & 1 & 1 & 1 & 0 & \cdots & \vdots \\ 
	0 & 1 & 1 & 1 & 1 & \ddots & \vdots \\
	\vdots & \ddots  & \ddots & \ddots &  \ddots & \ddots &  0 \\
	\vdots & \cdots & \ddots & 1 & 1 & 1 & 1 \\
	0 & \cdots & \cdots & 0 & 0 & 1 & 1 \\
	\end{vmatrix}_{n-2} + (-1)^{n}
	\begin{vmatrix}
	1 & 1 & 1 & 0 & \cdots & \cdots & 0  \\
	1 & 1 & 1 & 1 & \ddots  & \cdots & \vdots \\
	0 & 1 & 1 & 1 & 1 & \ddots & \vdots  \\ 
	\vdots & 0 & 1 & 1 & 1 & \ddots & 0 \\
	\vdots & \ddots  & \ddots & \ddots &  \ddots & \ddots &  1\\
	\vdots & \cdots & \ddots & 0 & 1 & 1 & 1  \\
	0 & \cdots & \cdots & \cdots & 0 & 0 & 1  \\
	\end{vmatrix}_{n-2} \\
	&=:  P_{n-1,1} + (-1)^{n}P_{n-1,2}.
	\end{aligned} \notag
	\end{equation}
	Notice that, applying the last column expansion,
	\begin{equation}
	\begin{aligned}
	P_{n-1,1} &= 
	\begin{vmatrix}
	1 & 1 & 0 & \cdots & \cdots & \cdots & 0 \\
	1 & 1 & 1 & 0 & \cdots & \cdots & \vdots \\
	1 & 1 & 1 & 1 & 0 & \cdots & \vdots \\ 
	0 & 1 & 1 & 1 & 1 & \ddots & \vdots \\
	\vdots & \ddots  & \ddots & \ddots &  \ddots & \ddots &  0 \\
	\vdots & \cdots & \ddots & 1 & 1 & 1 & 1 \\
	0 & \cdots & \cdots & 0 & 1 & 1 & 1 \\
	\end{vmatrix}_{n-3}  - 
	\begin{vmatrix}
	1 & 1 & 0 & \cdots & \cdots & \cdots & 0 \\
	1 & 1 & 1 & 0 & \cdots & \cdots & \vdots \\
	1 & 1 & 1 & 1 & 0 & \cdots & \vdots \\ 
	0 & 1 & 1 & 1 & 1 & \ddots & \vdots \\
	\vdots & \ddots  & \ddots & \ddots &  \ddots & \ddots &  0 \\
	\vdots & \cdots & \ddots & 1 & 1 & 1 & 1 \\
	0 & \cdots & \cdots & 0 & 0 & 0 & 1 \\
	\end{vmatrix}_{n-3}  \\
	&= 	C_{n-3} - C_{n-4}.
	\end{aligned} \notag
	\end{equation}
	Furthermore, by applying the last row expansion, 
	$$P_{n-1,2} = C_{n-3}.$$
	Combining all results,
	\begin{eqnarray*}
		K_{n} &=& R_{n-1} - P_{n-1} \\
		&=&  C_{n-2} + (-1)^{n}C_{n-2} - (C_{n-3} - C_{n-4}  + (-1)^{n}C_{n-3}) \\
		&=& (C_{n-2} - C_{n-3})(1 + (-1)^{n}) + C_{n-4}.
	\end{eqnarray*}
\end{proof}

\begin{lemma} \label{lem:ln}
	Let
	\begin{equation}
	L_{n}= 
	\begin{vmatrix}
	1 & 0 & 0 &   \cdots & 0  & 1  & 0\\
	1 & 1 & 0 &   \ddots & \ddots &  \ddots & 1  \\
	1 & 1 & 1 &    1 & \ddots  & \ddots & 0 \\ 
	0 & 1 & 1 &    \ddots & \ddots & \ddots & \vdots \\
	\vdots & \ddots  & \ddots & \ddots &    \ddots &  \ddots &  0\\
	\vdots &   & \ddots &  1 & 1 & 1 & 1\\
	0 & \cdots & \cdots &   0  & 1 & 1 & 1\\
	\end{vmatrix}_n. \notag
	\end{equation}
	Then 
	\begin{equation*}
		L_{n} = C_{n-2}(1+ (-1)^{n}) + (-1)^{n}(C_{n-3}-2C_{n-4}+C_{n-5}) + 1.
	\end{equation*}

\end{lemma}

\begin{proof}
	Expansion across the first row leads to
	\begin{equation}
	\begin{aligned}
	L_{n} &= \begin{vmatrix}
	1 & 0 & 0 &   \cdots & 0  & 0  & 1\\
	1 & 1 & 1 &   \ddots & \ddots &  \ddots & 0  \\
	1 & 1 & 1 &    1 & \ddots  & \ddots & 0 \\ 
	0 & 1 & 1 &    \ddots & \ddots & \ddots & \vdots \\
	\vdots & \ddots  & \ddots & \ddots &    \ddots &  \ddots &  0\\
	\vdots &   & \ddots &  1 & 1 & 1 & 1\\
	0 & \cdots & \cdots &   0  & 1 & 1 & 1\\
	\end{vmatrix}_{n-1}
	+ (-1)^{n}\begin{vmatrix}
	1 & 1 & 0 & 0 &  \cdots & 0 & 1  \\
	1 & 1 & 1 & 1 & \cdots & \ddots & 0\\
	0 & 1 & 1 & 1 & \ddots & \vdots  & \vdots \\ 
	\vdots  &  \ddots & \ddots   &    \ddots  & \ddots&  1 & 0 \\
	\vdots & \ddots  & \ddots & \ddots &  \ddots &  1 & 0\\
	\vdots & \cdots & \ddots & \ddots & 1 & 1 & 1 \\
	0 & \cdots & \cdots & \cdots & 0 & 1  & 1 \\
	\end{vmatrix}_{n-1} \\
	&=: D_{n-1} + (-1)^{n} E_{n-1}.
	\end{aligned} \notag
	\end{equation}
	\newline
	Expansion of $D_{n,1}$ across the first row gives us
	\begin{equation}
	\begin{aligned}
	D_{n-1} &= \begin{vmatrix}
	1 & 1 & 0 & \cdots & \cdots & \cdots & 0 \\
	1 & 1 & 1 & 0 & \cdots & \cdots & \vdots \\
	1 & 1 & 1 & 1 & 0 & \cdots & \vdots \\ 
	0 & 1 & 1 & 1 & 1 & \ddots & \vdots \\
	\vdots & \ddots  & \ddots & \ddots &  \ddots & \ddots &  0 \\
	\vdots & \cdots & \ddots & 1 & 1 & 1 & 1 \\
	0 & \cdots & \cdots & 0 & 1 & 1 & 1 \\
	\end{vmatrix}_{n-2}
	+ (-1)^{n} \begin{vmatrix}
	1 & 1 & 1 & 0 & \cdots & \cdots & 0 \\
	1 & 1 & 1 & 1 & \ddots & \cdots & \vdots \\
	0 & 1 & 1 & 1 & 1 & \ddots & \vdots \\ 
	\vdots &  \ddots & \ddots & \ddots &  \ddots & \ddots &  0 \\
	\vdots & \cdots & \ddots & 1 & 1 & 1 & 1 \\
	\vdots & \cdots & \cdots & \ddots & 1 & 1 & 1 \\
	0 & \cdots & \cdots & 0 & 0 & 1 & 1 \\
	\end{vmatrix}_{n-2} \\ 
	&= 
	C_{n-2} + (-1)^{n}C_{n-2}.
	\end{aligned} \notag
	\end{equation}
	For $E_{n-1}$, applying expansion along the first column, we obtain 
	\begin{equation}
	\begin{aligned}
	E_{n-1} &= \begin{vmatrix}
	1 & 1 & 1 & 0 &  \cdots & 0 & 0  \\
	1 & 1 & 1 & 1 & \cdots & \ddots & 0\\
	0 & 1 & 1 & 1 & \ddots & \vdots  & \vdots \\ 
	\vdots  &  \ddots & \ddots   &    \ddots  & \ddots&  1 & 0 \\
	\vdots & \ddots  & \ddots & \ddots &  \ddots &  1 & 0\\
	\vdots & \cdots & \ddots & \ddots & 1 & 1 & 1 \\
	0 & \cdots & \cdots & \cdots & 0 & 1  & 1 \\
	\end{vmatrix}_{n-2} 
	-  \begin{vmatrix}
	1 & 0 & 0 & 0 &  \cdots & 0 & 1  \\
	1 & 1 & 1 & 1 & \cdots & \ddots & 0\\
	0 & 1 & 1 & 1 & \ddots & \vdots  & \vdots \\ 
	\vdots  &  \ddots & \ddots   &    \ddots  & \ddots&  1 & 0 \\
	\vdots & \ddots  & \ddots & \ddots &  \ddots &  1 & 0\\
	\vdots & \cdots & \ddots & \ddots & 1 & 1 & 1 \\
	0 & \cdots & \cdots & \cdots & 0 & 1  & 1 \\
	\end{vmatrix}_{n-2} \\
	&=: F_{n-2} - H_{n-2}.
	\end{aligned} \notag
	\end{equation}
	The two terms on the right-hand side are computed as follows:
	\begin{equation}
	\begin{aligned}
	F_{n-2} &= \begin{vmatrix}
	1 & 1 & 1 & 0 &  \cdots & 0 & 0  \\
	1 & 1 & 1 & 1 & \cdots & \ddots & 0\\
	0 & 1 & 1 & 1 & \ddots & \vdots  & \vdots \\ 
	\vdots  &  \ddots & \ddots   &    \ddots  & \ddots&  1 & 0 \\
	\vdots & \ddots  & \ddots & \ddots &  \ddots &  1 & 1\\
	\vdots & \cdots & \ddots & \ddots & 1 & 1 & 1 \\
	0 & \cdots & \cdots & \cdots & 0 & 1  & 1 \\
	\end{vmatrix}_{n-3} 
	- \begin{vmatrix}
	1 & 1 & 1 & 0 &  \cdots & 0 & 0  \\
	1 & 1 & 1 & 1 & \cdots & \ddots & 0\\
	0 & 1 & 1 & 1 & \ddots & \vdots  & \vdots \\ 
	\vdots  &  \ddots & \ddots   &    \ddots  & \ddots&  1 & 0 \\
	\vdots & \ddots  & \ddots & \ddots &  \ddots &  1 & 1\\
	\vdots & \cdots & \ddots & \ddots & 1 & 1 & 1 \\
	0 & \cdots & \cdots & \cdots & 0 & 0 & 1 \\
	\end{vmatrix}_{n-3} \\
	&= C_{n-3} -C_{n-4},
	\end{aligned} \notag
	\end{equation}
	after expansion along the last column. Next, expanding $H_{n-2}$ across the first row,
	\begin{equation}
	\begin{aligned}
	H_{n-2}&= \begin{vmatrix}
	1 & 1 & 1 & 0 &  \cdots & 0 & 0  \\
	1 & 1 & 1 & 1 & \ddots & \ddots & 0\\
	0 & 1 & 1 & 1 & \ddots & \ddots  & \vdots \\ 
	\vdots  &  \ddots & \ddots   &    \ddots  & \ddots&  1 & 0 \\
	\vdots & \ddots  & \ddots & \ddots &  \ddots &  1 & 0\\
	\vdots & \ddots & \ddots & \ddots & 1 & 1 & 1 \\
	0 & \cdots & \cdots & \cdots & 0 & 1  & 1 \\
	\end{vmatrix}_{n-3}
	+ (-1)^{n-1}\begin{vmatrix}
	1 & 1 & 1 & 0 &  \cdots & 0 & 0  \\
	0 & 1 & 1 & 1 & \ddots & \ddots & 0\\
	\vdots & \ddots & 1 & 1 & \ddots & \ddots  & \vdots \\ 
	\vdots  &  \ddots & \ddots   &    \ddots  & \ddots&  1 & 0 \\
	\vdots & \ddots  & \ddots & \ddots &  \ddots &  1 & 1\\
	\vdots & \ddots & \ddots & \ddots & \ddots & 1 & 1 \\
	0 & \cdots & \cdots & \cdots & \cdots & 0 & 1 \\
	\end{vmatrix}_{n-3} \\ 
	&=Q_{n-3} + (-1)^{n-1}.
	\end{aligned} \notag
	\end{equation}
	Finally, using the last column $Q_{n-3}$,
	\begin{equation}
	\begin{aligned}
	Q_{n-3} &= \begin{vmatrix}
	1 & 1 & 1 & 0 &  \cdots & 0 & 0  \\
	1 & 1 & 1 & 1 & \cdots & \ddots & 0\\
	0 & 1 & 1 & 1 & \ddots & \vdots  & \vdots \\ 
	\vdots  &  \ddots & \ddots   &    \ddots  & \ddots&  1 & 0 \\
	\vdots & \ddots  & \ddots & \ddots &  \ddots &  1 & 1\\
	\vdots & \cdots & \ddots & \ddots & 1 & 1 & 1 \\
	0 & \cdots & \cdots & \cdots & 0 & 1  & 1 \\
	\end{vmatrix}_{n-4} 
	- \begin{vmatrix}
	1 & 1 & 1 & 0 &  \cdots & 0 & 0  \\
	1 & 1 & 1 & 1 & \cdots & \ddots & 0\\
	0 & 1 & 1 & 1 & \ddots & \vdots  & \vdots \\ 
	\vdots  &  \ddots & \ddots   &    \ddots  & \ddots&  1 & 0 \\
	\vdots & \ddots  & \ddots & \ddots &  \ddots &  1 & 1\\
	\vdots & \cdots & \ddots & \ddots & 1 & 1 & 1 \\
	0 & \cdots & \cdots & \cdots & 0 & 0 & 1 \\
	\end{vmatrix}_{n-4} \\
	&= C_{n-4} -C_{n-5}.
	\end{aligned} \notag
	\end{equation}
	Thus, $H_{n-2}=C_{n-4} - C_{n-5} + (-1)^{n-1}$.
	
	Combining the results, we have
	\begin{eqnarray*}
		E_{n-1} &=& F_{n-2}-H_{n-2}= C_{n-3} - C_{n-4} - (C_{n-4}-C_{n-5} + (-1)^{n-1})  \\
		&=& C_{n-3} - 2C_{n-4} + C_{n-5} + (-1)^{n}.
	\end{eqnarray*}
	Hence,		
	$$ L_{n} = C_{n-2}(1+ (-1)^{n}) + (-1)^{n}(C_{n-3}-2C_{n-4}+C_{n-5}) + 1.$$
\end{proof}

We are now in the position to prove the theorem. Expansion along the first column of the matrix in~\eqref{matM} and applying Lemma \ref{lem:kn} and \ref{lem:ln} we have  

\begin{equation}
\begin{aligned}		
M_{n} &= \begin{vmatrix}
1 & 1 & 0 & \cdots & \cdots & \cdots & 1 \\
1 & 1 & 1 & 0 & \cdots & \cdots & \vdots \\
1 & 1 & 1 & 1 & 0 & \cdots & \vdots \\ 
0 & 1 & 1 & 1 & 1 & \ddots & \vdots \\
\vdots & \ddots  & \ddots & \ddots &  \ddots & \ddots &  0 \\
\vdots & \cdots & \ddots & 1 & 1 & 1 & 1 \\
0 & \cdots & \cdots & 0 & 1 & 1 & 1 \\
\end{vmatrix}_{n-1}
- \begin{vmatrix}
1 & 0 & 0 & \cdots & \cdots & 1 & 0 \\
1 & 1 & 1 & 0 & \cdots & \cdots & \vdots \\
1 & 1 & 1 & 1 & 0 & \cdots & \vdots \\ 
0 & 1 & 1 & 1 & 1 & \ddots & \vdots \\
\vdots & \ddots  & \ddots & \ddots &  \ddots & \ddots &  0 \\
\vdots & \cdots & \ddots & 1 & 1 & 1 & 1 \\
0 & \cdots & \cdots & 0 & 1 & 1 & 1 \\
\end{vmatrix}_{n-1}\\
&+
\begin{vmatrix}
1 & 0 & 0 & \cdots & \cdots & 1 & 0 \\
1 & 1 & 0 & 0 & \cdots & \cdots & 1 \\
1 & 1 & 1 & 1 & 0 & \cdots & 0 \\ 
0 & 1 & 1 & 1 & 1 & \ddots & \vdots \\
\vdots & \ddots  & \ddots & \ddots &  \ddots & \ddots &  0 \\
\vdots & \cdots & \ddots & 1 & 1 & 1 & 1 \\
0 & \cdots & \cdots & 0 & 1 & 1 & 1 \\
\end{vmatrix}_{n-1} \\
&= A_{n-1} - K_{n-1} + L_{n-1} \\
&= C_{n-1} +(-1)^{n}C_{n-2} - (C_{n-3}-C_{n-4})(1+(-1)^{n}) - C_{n-5} \\
&+ C_{n-3}(1+ (-1)^{n-1})+ (-1)^{n-1}(C_{n-4}-2C_{n-5}+C_{n-6}) + 1 \\
&= C_{n}(1 + 2 (-1)^{n-1}) - 2(-1)^{n-1}C_{n-1}+1.
\end{aligned} \notag
\end{equation}
with the last line above obtained after using Lemma  \ref{lem:main}. Use of the property of $C_{n}$  (cf. the last paragraph of Section 2.1)  completes the proof.

\section{Generalization}

In this section, we give a proof for a more general result than that of Theorem~\ref{thm:conj1} and~\ref{thm:conj2}. The proof technique is again elementary. We first state our general result in Theorem~\ref{thm:general}.

\begin{theorem} \label{thm:general}
Let $B_n$ be the determinant of the matrix \eqref{genmatrix}, with  $n \ge 5$; i.e.,   
	\begin{equation*}
	B_{n} =
	\begin{vmatrix}
	1 & 1 & 0 &  \cdots & 0 & a & b\\
	1 & 1 & 1 & \ddots & \ddots & 0 & a \\
	1 & 1 & 1 & 1 &  \ddots & \ddots & 0 \\
	0 & 1 & 1 & 1 & \ddots & \ddots  & \vdots\\
	\vdots & \ddots & \ddots & \ddots &  \ddots & \ddots &  0 \\
	\vdots & \ddots & \ddots & 1 & 1 & 1 & 1 \\
	0 & \cdots & \cdots & 0 & 1 & 1 & 1 \\
	\end{vmatrix}_n. \notag
	\end{equation*}
	Then 
	\begin{equation*}
	B_n =
	\begin{cases}
	(a-1)^{2}, &\text{  if  } n \equiv 0 \Mod{4}, \\ 
	a^{2} + b + 1,  &\text{  if  } n \equiv 1 \Mod{4}, \\  
	a^{2} + 2a - b, &\text{  if  } n \equiv 2 \Mod{4}, \\
	a^{2},  &\text{  if  } n \equiv 3 \Mod{4}.
	\end{cases}
	\end{equation*}	
\end{theorem}


To prove the above theorem, we shall start with the following lemma.
\begin{lemma}\label{lem:Gn}
	Let
	\begin{equation*}
	G_{n} =
	\begin{vmatrix}
	1 & 1 & 0 & \cdots & \cdots & 0  & a\\
	1 & 1 & 1 & 1 & \ddots & \ddots & 0  \\
	0 & 1 & 1 & 1 & \ddots & \ddots & \vdots \\ 
	\vdots & \ddots & \ddots &  \ddots & \ddots & \ddots & 0\\
	\vdots &  \ddots & \ddots &\ddots &  \ddots &  1 &  1\\
	\vdots & \ddots & \ddots & \ddots & 1 & 1 & 1\\
	0 & \cdots & \cdots & \cdots  & 0 & 1 & 1\\
	\end{vmatrix}_n. 	
	\end{equation*}
	Then $$G_{n} = C_{n-1} - C_{n-2} + a(-1)^{n+1}.$$
\end{lemma}

\begin{proof}
	Expansion across the first row leads to
	\begin{equation*}
	\begin{aligned}
	G_{n} &=
	\begin{vmatrix}
	1 & 1 & 1 & 0 & \cdots & \cdots  & 0\\
	1 & 1 & 1 & 1 & \ddots & \ddots & \vdots \\
	0 & 1 & 1 & 1 & \ddots & \ddots & \vdots \\ 
	\vdots & \ddots & \ddots &  \ddots & \ddots & \ddots & 0\\
	\vdots &  \ddots  & \ddots &\ddots &  \ddots &  1 &  1\\
	\vdots & \ddots & \ddots & \ddots & 1 & 1 & 1\\
	0 & \cdots & \cdots & \cdots  & 0 & 1 & 1\\
	\end{vmatrix}_{n-1} -
	\begin{vmatrix}
	1 & 1 & 1 & 0 & \cdots & \cdots  & 0\\
	0 & 1 & 1 & 1 & \ddots & \ddots & \vdots \\
	0 & 1 & 1 & 1 & \ddots & \ddots & \vdots \\ 
	\vdots & \ddots & \ddots &  \ddots & \ddots & \ddots & 0\\
	\vdots &  \ddots & \ddots &\ddots &  \ddots &  1 &  1\\
	\vdots & \ddots & \ddots & \ddots & 1 & 1 & 1\\
	0 & \cdots & \cdots & \cdots  & 0 & 1 & 1\\
	\end{vmatrix}_{n-1} \\
	&+a(-1)^{n+1}
	\begin{vmatrix}
	1 & 1 & 1 & 1 &  0 & \cdots  & 0\\
	0 & 1 & 1 & 1 & \ddots & \ddots & \vdots \\
	\vdots  & \ddots & 1 & 1 & \ddots & \ddots & 0 \\ 
	\vdots & \ddots & \ddots &  \ddots & \ddots & \ddots & 1\\
	\vdots &  \ddots  & \ddots &\ddots &  \ddots &  1 &  1\\
	\vdots &\ddots & \ddots & \ddots & \ddots & 1 & 1\\
	0 & \cdots & \cdots & \cdots  & \cdots & 0 & 1\\
	\end{vmatrix}_{n-1} \\
	&= C_{n-1} - C_{n-2} + (-1)^{n+1} a,
	\end{aligned}
	\end{equation*}
	after further expansion of the middle term across the first column and using the fact that $|\cdot^T| = |\cdot|$.
\end{proof}

\begin{lemma}\label{lem:Tn}
	Let 
	\begin{equation*}
	T_{n} =
	\begin{vmatrix}
	1 & 1 & 0 & \cdots &  \cdots & 0  & a\\
	1 & 1 & 1 & \ddots & \ddots & \ddots &  0\\
	1  & 1 & 1 & 1 & \ddots & \ddots & \vdots \\ 
	0 & \ddots & \ddots &  \ddots & \ddots & \ddots & \vdots \\
	\vdots &  \ddots  & \ddots &\ddots &  \ddots &  1 &  0\\
	\vdots & \ddots & \ddots & 1 & 1 & 1 & 1\\
	0 & \cdots & \cdots & 0  & 0 & 1 & 1\\
	\end{vmatrix}_n
	\end{equation*}
	Then $$T_{n} = C_{n-2}(1 + a(-1)^{n+1}) - 2C_{n-3} +2C_{n-4} - C_{n-5},$$
	 for $n \ge 6$.
\end{lemma}

\begin{proof}
	Expansion of $T_{n}$ across the first row yields
	\begin{equation*}
	\begin{aligned}
	T_{n} &=
	\begin{vmatrix}
	1 & 1 & 0 & \cdots &  \cdots & 0  & 0\\
	1 & 1 & 1 & \ddots & \ddots & \ddots &  0\\
	1  & 1 & 1 & 1 & \ddots & \ddots & \vdots \\ 
	0 & \ddots & \ddots &  \ddots & \ddots & \ddots & \vdots \\
	\vdots &  \ddots  & \ddots &\ddots &  \ddots &  1 &  0\\
	\vdots & \cdots & \ddots & 1 & 1 & 1 & 1\\
	0 & \cdots & \cdots & 0  & 0 & 1 & 1\\
	\end{vmatrix}_{n-1}-
	\begin{vmatrix}
	1 & 1 & 0 & \cdots &  \cdots & 0  & 0\\
	1 & 1 & 1 & \ddots & \ddots & \ddots &  0\\
	0 & 1 & 1 & 1 & \ddots & \ddots & \vdots \\ 
	0 & 1 & 1 &  \ddots & \ddots & \ddots & \vdots \\
	\vdots &  \ddots  & \ddots &\ddots &  \ddots &  1 &  0\\
	\vdots & \cdots & \ddots & 1 & 1 & 1 & 1\\
	0 & \cdots & \cdots & 0  & 0 & 1 & 1\\
	\end{vmatrix}_{n-1} \\
	&+ a(-1)^{n+1}
	\begin{vmatrix}
	1 & 1 & 1 & 0  &  \cdots & 0  & 0\\
	1 & 1 & 1 &  1& \ddots & \ddots &  0\\
	0 & 1 & 1 & 1 & \ddots & \ddots & \vdots \\ 
	\vdots & \ddots & 1 &  \ddots & \ddots & \ddots & 0 \\
	\vdots &  \ddots  & \ddots &\ddots &  \ddots &  1 &  1\\
	\vdots & \ddots & \ddots &\ddots  & 1 & 1 & 1\\
	0 & \cdots & \cdots & \cdots  & 0 & 0 & 1\\
	\end{vmatrix}_{n-1}  \\
	&=: M_{n-1} - N_{n-1} + a(-1)^{n+1}V_{n-1}.
	\end{aligned}
	\end{equation*} 
	$M_{n-1}$ can be expressed in terms of $C_n$ after the expansion along the last column as follows:
	\begin{equation*}
	\begin{aligned}
	M_{n-1} &= \begin{vmatrix}
	1 & 1 & 0 & \cdots  &  \cdots & 0  & 0\\
	1 & 1 & 1 &  \ddots & \ddots & \ddots &  0\\
	1 & 1 & 1 & 1 & \ddots & \ddots & \vdots \\ 
	0 & \ddots & 1 &  \ddots & \ddots & \ddots & \vdots \\
	\vdots &  \ddots  & \ddots &\ddots &  \ddots &  1 & 0 \\
	\vdots & \ddots & \ddots & 1  & 1 & 1 & 1\\
	0 & \cdots & \cdots & 0  & 1 & 1 & 1\\
	\end{vmatrix}_{n-2} - 
	\begin{vmatrix}
	1 & 1 & 0 & \cdots  &  \cdots & 0  & 0\\
	1 & 1 & 1 &  \ddots & \ddots & \ddots &  0\\
	1 & 1 & 1 & 1 & \ddots & \ddots & \vdots \\ 
	0 & \ddots & 1 &  \ddots & \ddots & \ddots & \vdots \\
	\vdots &  \ddots  & \ddots &\ddots &  \ddots &  1 & 0 \\
	\vdots & \ddots & \ddots & 1  & 1 & 1 & 1\\
	0 & \cdots & \cdots & 0  & 0 & 0 & 1\\
	\end{vmatrix}_{n-2} \\
	&= C_{n-2} - C_{n-3}.
	\end{aligned}
	\end{equation*}
	For $N_{n-1}$, applying expansion along the first column, 
	\begin{equation*}
	\begin{aligned}
	N_{n-1} &= \begin{vmatrix}
	1 & 1 & 0 & \cdots  &  \cdots & 0  & 0\\
	1 & 1 & 1 &  \ddots & \ddots & \ddots &  0\\
	1 & 1 & 1 & 1 & \ddots & \ddots & \vdots \\ 
	0 & \ddots & 1 &  \ddots & \ddots & \ddots & \vdots \\
	\vdots &  \ddots  & \ddots &\ddots &  \ddots &  1 & 0 \\
	\vdots & \ddots & \ddots & 1  & 1 & 1 & 1\\
	0 & \cdots & \cdots & 0  & 0 & 1 & 1\\
	\end{vmatrix}_{n-2} -
	\begin{vmatrix}
	1 & 0 & 0 & \cdots  &  \cdots & 0  & 0\\
	1 & 1 & 1 &  \ddots & \ddots & \ddots &  0\\
	1 & 1 & 1 & 1 & \ddots & \ddots & \vdots \\ 
	0 & \ddots & 1 &  \ddots & \ddots & \ddots & \vdots \\
	\vdots &  \ddots  & \ddots &\ddots &  \ddots &  1 & 0 \\
	\vdots & \ddots & \ddots & 1  & 1 & 1 & 1\\
	0 & \cdots & \cdots & 0  & 0 & 1 & 1\\
	\end{vmatrix}_{n-2} \\
	&= : M_{n-2} - N_{n-2, 2}.
	\end{aligned}\\
	\end{equation*}
	Notice that, after expansion across the first row, $N_{n-2,2} = M_{n-3} $.  Finally, applying the last row expansion to $V_{n-1}$ and using the fact that $|\cdot^T| = |\cdot|$, we have $V_{n-1} = C_{n-2}$.
	Combining all results, 
	\begin{equation*}
	\begin{aligned}
	T_{n} &= (C_{n-2}-C_{n-3}) - (C_{n-3} - C_{n-4}) + (C_{n-4} - C_{n-5}) + a(-1)^{n+1}C_{n-2}\\
	&=C_{n-2}(1 + a(-1)^{n+1}) - 2C_{n-3} +2C_{n-4} - C_{n-5}.
	\end{aligned}
	\end{equation*}
\end{proof}

We now state our last lemma.

\begin{lemma}\label{lem:Pn}
	Let 
	\begin{equation*}
	P_{n} =
	\begin{vmatrix}
	1 & 1 & 0 & 0  &  \cdots & 0  & a\\
	1 & 1 & 1 &  \ddots & \ddots & \ddots &  0\\
	1 & 1 & 1 & 1 & \ddots & \ddots & \vdots \\ 
	0 & \ddots & \ddots &  \ddots & \ddots & \ddots & \vdots \\
	\vdots &  \ddots  & \ddots &\ddots &  \ddots &  \ddots & 0 \\
	\vdots & \ddots & \ddots & 1  & 1 & 1 & 1\\
	0 & \cdots & \cdots & 0  & 1 & 1 & 1\\
	\end{vmatrix}_n.
	\end{equation*}
	Then $$P_{n} = C_{n-1}(1+a(-1)^{n+1}) - C_{n-2} + C_{n-3}.$$
\end{lemma}

\begin{proof}
	Expansion of $P_n$ across the first row leads to
	\begin{equation*}
	\begin{aligned}
	P_{n} &= \begin{vmatrix}
	1 & 1 & 0 & \cdots  &  \cdots & 0  & 0\\
	1 & 1 & 1 &  \ddots & \ddots & \ddots &  0\\
	1 & 1 & 1 & 1 & \ddots & \ddots & \vdots \\ 
	0 & \ddots & \ddots &  \ddots & \ddots & \ddots & \vdots \\
	\vdots &  \ddots  & \ddots &\ddots &  \ddots &  \ddots & 0 \\
	\vdots & \ddots & \ddots & 1  & 1 & 1 & 1\\
	0 & \cdots & \cdots & 0  & 1 & 1 & 1\\
	\end{vmatrix}_{n-1} -
	\begin{vmatrix}
	1 & 1 & 0 & \cdots  &  \cdots & 0  & 0\\
	1 & 1 & 1 &  \ddots & \ddots & \ddots &  0\\
	0 & 1 & 1 & 1 & \ddots & \ddots & \vdots \\ 
	0 & 1 & \ddots &  \ddots & \ddots & \ddots & \vdots \\
	\vdots &  \ddots  & \ddots &\ddots &  \ddots &  \ddots & 0 \\
	\vdots & \ddots & \ddots & 1  & 1 & 1 & 1\\
	0 & \cdots & \cdots & 0  & 1 & 1 & 1\\
	\end{vmatrix}_{n-1}\\
	&+a(-1)^{n+1}\begin{vmatrix}
	1 & 1 & 1 & 0   &  \cdots & 0  & 0\\
	1 & 1 & 1 &  1 & \ddots & \ddots &  0\\
	0 & 1 & 1 & 1 & \ddots & \ddots & \vdots \\ 
	\vdots  & \ddots & \ddots &  \ddots & \ddots & \ddots & 0 \\
	\vdots &  \ddots  & \ddots &\ddots &  \ddots &  \ddots & 1 \\
	\vdots & \ddots & \ddots & \ddots  & 1 & 1 & 1\\
	0 & \cdots & \cdots & \cdots  & 0 & 1 & 1\\
	\end{vmatrix}_{n-1} \\
	&=: C_{n-1} - H_{n-1} + a(-1)^{n+1}C_{n-1},
	\end{aligned}
	\end{equation*}
	where we have used $|\cdot^T| = |\cdot|$.
	
	Applying the first column expansion for $H_{n-1}$,
	\begin{equation*}
	\begin{aligned}
	H_{n-1}&= \begin{vmatrix}
	1 & 1 & 0 & \cdots   &  \cdots & 0  & 0\\
	1 & 1 & 1 &  \ddots & \ddots & \ddots &  0\\
	1 & 1 & 1 & 1 & \ddots & \ddots & \vdots \\ 
	0  & \ddots & \ddots &  \ddots & \ddots & \ddots & \vdots \\
	\vdots &  \ddots  & \ddots &\ddots &  \ddots &  \ddots & 0 \\
	\vdots & \ddots & \ddots & 1  & 1 & 1 & 1\\
	0 & \cdots & \cdots & 0   &  1 & 1 & 1\\
	\end{vmatrix}_{n-2} -
	\begin{vmatrix}
	1 & 0 & 0 & \cdots   &  \cdots & 0  & 0\\
	1 & 1 & 1 &  \ddots & \ddots & \ddots &  0\\
	1 & 1 & 1 & 1 & \ddots & \ddots & \vdots \\ 
	0  & \ddots & \ddots &  \ddots & \ddots & \ddots & \vdots \\
	\vdots &  \ddots  & \ddots &\ddots &  \ddots &  \ddots & 0 \\
	\vdots & \ddots & \ddots & 1  & 1 & 1 & 1\\
	0 & \cdots & \cdots & 0   &  1 & 1 & 1\\
	\end{vmatrix}_{n-2} \\
	&= C_{n-2} - C_{n-3}.
	\end{aligned}
	\end{equation*}
	Combining the results, we have
	$$P_{n} = C_{n-1}(1+a(-1)^{n+1}) - C_{n-2} + C_{n-3}.$$
\end{proof}

We now are ready to prove Theorem~\ref{thm:general}.

Expanding $B_{n}$ along the last column, we have
	\begin{align*}
	B_{n} &= b(-1)^{n+1}
	\begin{vmatrix}
	1 & 1 & 1 & 0  &  \cdots & 0  & 0\\
	1 & 1 & 1 &  1 & \ddots & \ddots &  0\\
	0 & 1 & 1 & 1 & \ddots & \ddots & \vdots \\ 
	\vdots & \ddots & 1 &  \ddots & \ddots & \ddots & 0 \\
	\vdots &  \ddots  & \ddots &\ddots &  \ddots &  1 & 1 \\
	\vdots & \ddots & \ddots & \ddots  & 1 & 1 & 1\\
	0 & \cdots & \cdots & 0  & 0 & 1 & 1\\
	\end{vmatrix}_{n-1} + a(-1)^{n+2}
	\begin{vmatrix}
	1 & 1 & 0 & 0  &  \cdots & 0  & a\\
	1 & 1 & 1 &  1 & \ddots & \ddots &  0\\
	0 & 1 & 1 & 1 & \ddots & \ddots & \vdots \\ 
	\vdots & \ddots & 1 &  \ddots & \ddots & \ddots & 0 \\
	\vdots &  \ddots  & \ddots &\ddots &  \ddots &  1 & 1 \\
	\vdots & \ddots & \ddots & \ddots  & 1 & 1 & 1\\
	0 & \cdots & \cdots & 0  & 0 & 1 & 1\\
	\end{vmatrix}_{n-1}\\
	&-\begin{vmatrix}
	1 & 1 & 0 & 0  &  \cdots & 0  & a\\
	1 & 1 & 1 &  \ddots & \ddots & \ddots &  0\\
	1 & 1 & 1 & 1 & \ddots & \ddots & \vdots \\ 
	0 & \ddots & \ddots &  \ddots & \ddots & \ddots & 0 \\
	\vdots &  \ddots  & \ddots &\ddots &  \ddots &  1 & 0 \\
	\vdots & \ddots & \ddots & 1  & 1 & 1 & 1\\
	0 & \cdots & \cdots & 0  & 0 & 1 & 1\\
	\end{vmatrix}_{n-1} + 
	\begin{vmatrix}
	1 & 1 & 0 & 0  &  \cdots & 0  & a\\
	1 & 1 & 1 &  \ddots & \ddots & \ddots &  0\\
	1 & 1 & 1 & 1 & \ddots & \ddots & \vdots \\ 
	0 & \ddots & \ddots &  \ddots & \ddots & \ddots & 0 \\
	\vdots &  \ddots  & \ddots &\ddots &  \ddots &  1 & 0 \\
	\vdots & \ddots & \ddots & 1  & 1 & 1 & 1\\
	0 & \cdots & \cdots & 0  & 1 & 1 & 1\\
	\end{vmatrix}_{n-1} \\[1em]
	=&b(-1)^{n+1}C_{n-1}+ a(-1)^{n+2}G_{n-1} - T_{n-1} + P_{n-1}.
	\end{align*}
	By using Lemmas~\ref{lem:main},~\ref{lem:Gn},~\ref{lem:Tn} and~\ref{lem:Pn}, we get for $n \ge 7$
	$$
	B_{n} = 3C_{n} + C_{n-1}(b(-1)^{n+1}-2)+2(C_{n-2}-C_{n-3})(1 + a(-1)^{n}) + a^{2}.
	$$
	Using Lemma~\ref{lem:main} and $C_i$, given in the end of Section 2.1 lead to the result in Theorem~\ref{thm:general}, for $n \ge 7$. For $n=5,6$, the outcome of the theorem can be computed using the expansion across the first row of $B_n$.

\section{Numerical Examples}
\index{Numerical examples%
	@\emph{Numerical examples}}%

Finally, we conducted numerical experiment to verify the theorem results. The code was written in Matlab and the determinant was found by build-in function. We consider all natural numbers from $5$ to $10 000$, using the setting of $a$ and $b$ as in Theorems \ref{thm:conj1} and \ref{thm:conj2}. As can be seen from Fig. \ref{fig:detest}, there is excellent match between the numerical values of determinants ($A_n$ and $M_n$)  and outcomes of theorems.

\begin{figure}[H]
	\begin{subfigure}{.5\textwidth}
		\centering
		\includegraphics[width=1.0\linewidth]{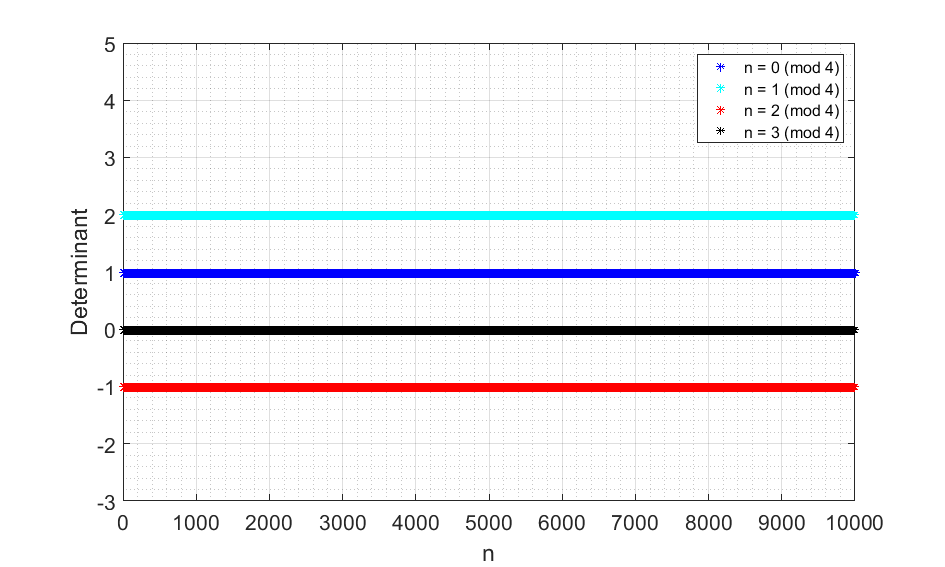}
		\caption{Determinant of $A_n$.}
	\end{subfigure}
	\begin{subfigure}{.5\textwidth}
		\centering
		\includegraphics[width=1.0\linewidth]{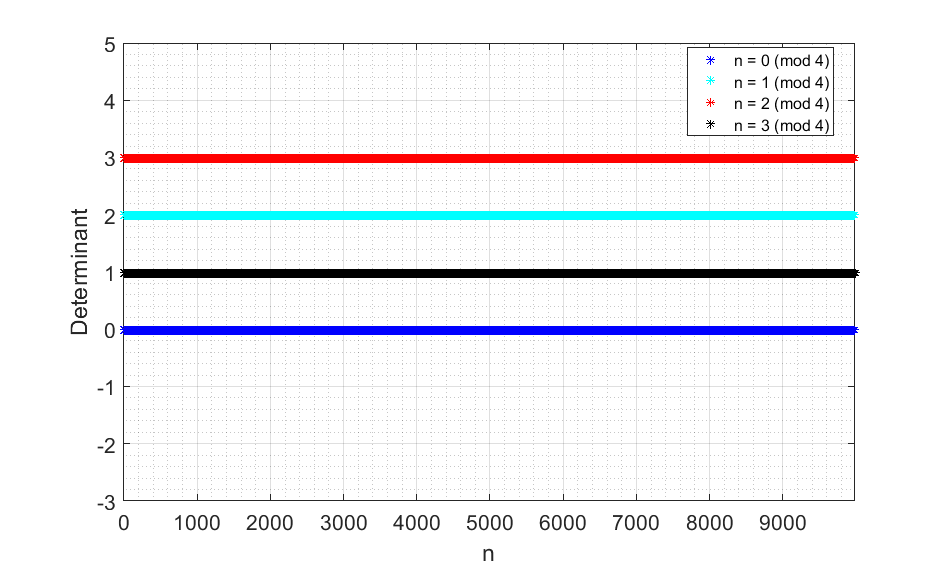}
		\caption{Determinant of $M_n$.}
	\end{subfigure}
	\caption{Numerical values of determinant.}
	\label{fig:detest}
\end{figure}

\section{Conclusion}

We have proved conjectures, which is related to the explicit formula of determinants of Toeplitz matrices, proposed in \cite{andjelic2020some}. The proof was also demonstrated for a general theorem. We have also presented numerical examples for the theorem results. Future work will involve determinant of Toeplitz matrices with arbitrary values where four subdiagonals are not consecutive.

\section*{Acknowledgements} 
The Nazarbayev University Faculty Development Competitive Research Grant, No 110119FD4502, is acknowledged. 
 
\bibliographystyle{unsrt}  
\bibliography{arXiv_proof_conjectures_det_toeplitz}

\end{document}